\documentclass[12pt,reqno]{amsart}
\usepackage{amsmath, amssymb, graphicx}
\usepackage[colorlinks=true]{hyperref}

\newtheorem{thm}{Theorem}
\newtheorem{lem}{Lemma}
\newtheorem{dfn}{Definition}
\newtheorem{prop}{Proposition}
\newtheorem{conj}{Conjecture}
\newtheorem{remark}{Remark}

\newcommand{\dx}{\mathrm{d}x}

\begin{document}

\title[Moments of Products of Elliptic Integrals]{Moments of Products of Elliptic Integrals}
\author[J. Wan]{James G. Wan}
\address{CARMA, University of Newcastle, Callaghan, NSW, 2308, Australia}
\email{james.wan@newcastle.edu.au}
\date{\today}
\maketitle

\begin{abstract}
We consider the moments of products of complete elliptic integrals of the first and second kinds. In particular, we derive new results using elementary means, aided by computer experimentation and a theorem of W. Zudilin. Diverse related evaluations, and two conjectures, are also given.
\end{abstract}

\medskip

\noindent{\footnotesize {\bf 2010 AMS Classification Numbers}: Primary 33C75; Secondary 33C20, 42A16, 60G50.} 

\medskip

\section{Motivation and general approach}

We study the complete elliptic integral of the first kind, $K(x)$, and the second kind, $E(x)$, defined by:
\begin{dfn}
\begin{equation}
K(x) = \frac{\pi}2 \, _2F_1\left( {{\frac12,\frac12}\atop{1}}\bigg| x^2\right), \quad
E(x) = \frac{\pi}2 \, _2F_1\left( {{-\frac12,\frac12}\atop{1}}\bigg| x^2\right).
\end{equation}
\end{dfn} 
As usual, $K'(x), E'(x)$ are used to denote the function value at the complementary modulus, $x' = \sqrt{1-x^2}$.
Recall that $_pF_q$ denotes the \textit{generalized hypergeometric series},
\begin{equation}
_pF_q\left( {{a_1, \ldots, a_p}\atop{b_1, \ldots, b_q}}\bigg| z \right)  = \sum_{n=0}^\infty \frac{(a_1)_n \cdots (a_p)_n}{(b_1)_n \cdots (b_q)_n} \frac{z^n}{n!}.
\end{equation}

The complete elliptic integrals, apart from their theoretical importance in arbitrary precision numerical computations (\cite{agm}) and the theory of theta functions, are also of significant interest in applied fields such as electrodynamics (\cite{wolf}), statistical mechanics, and random walks. Indeed, they were first used to provide explicit solutions to the perimeter of an ellipse (among other curves) as well as the (exact) period of an ideal pendulum.

The author was first drawn to the study of integral of products of $K$ and $E$ in \cite{walks2}, in which it is shown that 
\begin{equation} \label{1} 2 \int_0^1 K(x)^2 \, \dx = \int_0^1 K'(x)^2 \, \dx, \end{equation} by relating both sides to a moment of the distance from the origin in a four step uniform random walk on the plane. 

A much less recondite proof was only found later: set  $x=(1-t)/(1+t)$ on the left hand side of (\ref{1}), and apply the quadratic transform (\ref{k1}) below, and the result readily follows.

The four quadratic transforms (\cite{agm}), which we will use over and over again, are:
\begin{eqnarray}
K'(x) & = & \frac{2}{1+x} K\left(\frac{1-x}{1+x}\right) \label{k1} \\
K(x) & = & \frac{1}{1+x} K\left(\frac{2\sqrt{x}}{1+x}\right) \label{k2} \\
E'(x) & = & (1+x) E\left(\frac{1-x}{1+x}\right) - x K'(x) \label{e1} \\
E(x) & = & \frac{1+x}{2} E\left(\frac{2\sqrt x}{1+x}\right) + \frac{1-x^2}{2}K(x).  \label{e2}
\end{eqnarray}

In the following sections we will consider definite integrals involving products of $K, E, K', E'$, especially the moments of the products, with the intent of producing closed forms whenever possible.

The somewhat rich and unexpected results are perhaps surprisingly easy to discover, thanks to methods in experimental mathematics, for instance, the \textit{integer relations algorithm} PSLQ, the \textit{Inverse Symbolic Calculator} (ISC, now hosted at CARMA, \cite{isc}), the \textit{Online Encyclopedia of Integer Sequences} (OEIS, \cite{slo}), the \textit{Maple} package gfun, \textit{Gosper's algorithm} (which finds closed forms for indefinite sums of hypergeometric terms, \cite{ab}), and Sister Celine's method  (\cite{ab}).  Indeed, large scale computer experiments (\cite{combat})  revealed that there was a huge number of identities in the flavour of (\ref{1}). Once discovered, many results can be routinely established by the following elementary techniques:

\begin{enumerate}
\item Connections with and transforms of hypergeometric and Meijer G-functions (\cite{wolf}), as in the case of random walk integrals (Section \ref{two1}).
\item Interchange order of summation and integration (Section \ref{two2}).
\item Evaluate a closed form for the primitive (Section \ref{sporadic}).
\item Change the variable $x$ to $x'$ (Section \ref{sporadic}).
\item Change the variable followed by quadratic transforms (Section \ref{sporadic}).
\item Use a Fourier series (Section \ref{fs}).
\item Apply Legendre's relation (Section \ref{lege}).
\item Differentiate a product of elliptic integrals and integrate by parts (Section \ref{parts}).
\end{enumerate}

\section{One elliptic integral}

The moments of $K, E, K', E'$ are established, for instance, in \cite{agm}. Moreover, in \cite{catalan}, all moments of the \textit{generalized elliptic integrals} are similarly derived. Here we slightly generalize the results in \cite{agm}.

By a straight-forward application of the formulae (see \cite{catalan})

\[ \int_0^1x^{u-1}(1-x)^{v-1} \, _2F_1\left({{a,1-a}\atop{b}}\bigg|x\right) \, \dx = \frac{\Gamma(u)\Gamma(v)}{\Gamma(u+v)} \, _3F_2 \left({{a,1-a,u}\atop{b,u+v}}\bigg|1 \right), \]
\[ \int_0^1x^{u-1}(1-x)^{v-1} \, _2F_1\left({{a,-a}\atop{b}}\bigg|x\right) \, \dx = \frac{\Gamma(u)\Gamma(v)}{\Gamma(u+v)} \, _3F_2 \left({{a,-a,u}\atop{b,u+v}}\bigg|1 \right), \]
we produce:
\begin{equation} \label{s1}
 \int_0^1 x'^n x^m K(x) \, \dx = \frac{\pi}4  \frac{\Gamma(\frac12(m+1))\Gamma(\frac12(n+2))}{\Gamma(\frac12 (m+n+3))} \, _3F_2\left( {{\frac12, \frac12, \frac{m+1}2}\atop{1, \frac{m+n+3}2}}\bigg| 1 \right), 
\end{equation}
and
\begin{equation} \label{s2}
 \int_0^1 x'^n x^m E(x) \, \dx = \frac{\pi}4 \frac{\Gamma(\frac12(m+1))\Gamma(\frac12(n+2))}{\Gamma(\frac12 (m+n+3))} \, _3F_2\left( {{-\frac12, \frac12, \frac{m+1}2}\atop{1, \frac{m+n+3}2}}\bigg| 1 \right).
\end{equation}

These also encapsulate the moments for $K', E'$, because by $x \mapsto x'$, we have
\[ \int_0^1 x'^n x^m K'(x) \, \dx = \int_0^1 x^{n+1} x'^{m-1} K(x) \, \dx, \ \int_0^1 x'^n x^m E'(x) \, \dx = \int_0^1 x^{n+1} x'^{m-1} E(x) \, \dx. \]

We note that for convergence, $m>-1, n>-2$. When $m=1$ both formulae reduce to a $_2F_1$ and can be summed by Gauss' theorem (\cite{aar}).

If in addition $2m+n+1=0$ in (\ref{s1}), then Dixon's theorem (\cite{aar}) applies and we may sum the $_3F_2$ explicitly in terms of the $\Gamma$ function. For instance, we may compute $\int_0^1 K(x)/x' \, \dx$ (which also follows from the Fourier series in Section \ref{fs}). 

In (\ref{s2}), Dixon's theorem may only be applied to the single special case $\int_0^1 x'E(x) \, \dx$.

\section{One elliptic integral and one complementary elliptic integral} \label{two2}

Here we take advantage of the closed form for moments of $K', E'$ which  follow from (\ref{s1}) and (\ref{s2}):
\begin{eqnarray*}
\int_0^1 x^n K'(x) \, \dx & = & \frac{\pi \, \Gamma(\frac12(n+1))^2}{4 \, \Gamma(\frac12(n+2))^2}, \\
\int_0^1 x^n E'(x) \, \dx & = & \frac{\pi \, \Gamma(\frac12(n+3))^2}{2(n+1) \, \Gamma(\frac12(n+2))\Gamma(\frac12(n+4))}.
\end{eqnarray*}

We also use the (hypergeometric) series for $K, E$ which follow from the definitions:
\begin{eqnarray*}
K(x) & = & \sum_{k=0}^\infty  \frac{\Gamma(k+1/2)^2}{\Gamma(k+1)^2}\frac{x^{2k}}{2}, \\
E(x) & = & \sum_{k=0}^\infty -\frac{\Gamma(k-1/2)\Gamma(k+1/2)}{\Gamma(k+1)^2}\frac{x^{2k}}{4}.
\end{eqnarray*}

Hence, for the proposition below, we simply write the elliptic integral as a sum, and observe that the order of summation and integration may be interchanged.
\begin{prop} \label{prop1} We have the following moments:
\end{prop}
\begin{enumerate}

\item \begin{equation}\label{kk1} \int_0^1 x^n K(x)K'(x) \, \dx = \frac{\pi^2}8 \frac{\Gamma(\frac12(n+1))^2}{ \Gamma(\frac12(n+2))^2} \, _4F_3\left( {{\frac12,\frac12,\frac{n+1}2,\frac{n+1}2}\atop{1,\frac{n+2}2,\frac{n+2}2}}\bigg| 1\right). \end{equation}

In this case, odd $n$ yields a closed form in terms of $\pi^3$, and possibly a rational multiple of $\pi$, as we can expand the summand (which is a rational function) as a partial fraction much like in Remark \ref{partfrac}. Amazingly, all the rational multiples of $\pi$ vanish as we show in Lemma \ref{pi3}.

\item \begin{equation} \int_0^1 x^n E(x)K'(x) \, \dx = \frac{\pi^2}8 \frac{ \Gamma(\frac12(n+1))^2}{ \Gamma(\frac12(n+2))^2} \, _4F_3\left( {{-\frac12,\frac12,\frac{n+1}2,\frac{n+1}2}\atop{1,\frac{n+2}2,\frac{n+2}2}}\bigg| 1\right). \end{equation}

In this case, odd $n$ yields a closed form in terms of $\pi$ and $\pi^3$. From Lemma \ref{pi3} and (\ref{leg}), the constant of the $\pi$ term is $\frac{\pi}{4(n+1)}$.

\item \begin{equation} \int_0^1 x^n K(x)E'(x) \, \dx = \frac{\pi^2}8 \frac{(n+1) \Gamma(\frac12(n+1))^2}{(n+2) \Gamma(\frac12(n+2))^2} \, _4F_3\left( {{\frac12,\frac12,\frac{n+1}2,\frac{n+3}2}\atop{1,\frac{n+2}2,\frac{n+4}2}}\bigg| 1\right). \end{equation}

In this case, odd $n$ yields a closed form in terms of $\pi$ and $\pi^3$. From Lemma \ref{pi3} and (\ref{leg}), the constant of the $\pi$ term is $\frac{\pi}{4(n+1)}$.

\item \begin{equation} \int_0^1 x^n E(x)E'(x) \, \dx = \frac{\pi^2}8 \frac{(n+1) \Gamma(\frac12(n+1))^2}{(n+2) \Gamma(\frac12(n+2))^2} \, _4F_3\left( {{-\frac12,\frac12,\frac{n+1}2,\frac{n+3}2}\atop{1,\frac{n+2}2,\frac{n+4}2}}\bigg| 1\right). \end{equation}

In this case, odd $n$ yields a closed form in terms of $\pi$ and $\pi^3$. By using the derivative of $(1-x^2)x^{2n}E(x)E'(x)$ and the linear terms above, we deduce that the linear term here is also $\frac{\pi}{4(n+1)}$.

\end{enumerate}

We observe that $E', K'$ do not have `nice' expansions around the origin, and $E, K$ do not have `nice' moments, so the method above cannot to extended to evaluate moments of other products.

Note that by using the symmetry between parts 2 and 3, as well as by applying Legendre's relation (\ref{legendre}), we obtain linear identities connecting these $_4F_3$'s.

We now prove the observation in part 1 of Proposition \ref{prop1}.

\begin{lem} \label{pi3} When $n$ is odd, the $n$th moment of $K(x)K'(x)$ is a rational multiple of $\pi^3$. 
\end{lem}

\begin{proof}We experimentally discovered that, letting $g(n) := \int_0^1 x^{2n-1} K(x)K'(x) \, \dx$, we have
\[ 2n^3 g(n+1) - (2n-1)(2n^2-2n+1)g(n)+2(n-1)^3g(n-1)=0. \]
This (contiguous) relation, once discovered, can be proven by extracting the summand, simplifying and summing using Gosper's algorithm. Thus, after computing 2 starting values, the claim is proven. Note that the recursion also holds when $n$ is not an integer.
\end{proof}

Experimentally, we found that the sequence $h(n) := \pi^3 16^{n+1} g(n+1)$ matched entry \texttt{A036917} of the Online Encyclopedia of Integer Sequences; indeed, they share the same recursion and initial values. Moreover, the OEIS provides that 
\begin{equation} \label{altf}
h(n) = \sum_{k=0}^n \binom{2n-2k}{n-k}^2 \binom{2k}{k}^2 = \frac{16^n \Gamma(n+1/2)^2}{\pi \Gamma(n+1)^2} \, _4F_3\left( {{-n,-n,\frac12,\frac12}\atop{\frac12-n,\frac12-n,1}}\bigg| 1\right).
\end{equation}
The first equality is routine to check as we can  produce a recurrence for the binomial sum -- for instance, using Sister Celine's method; the second equality is notational.

The generating function for $h(n)$ is simply
\[ \sum_{n=0}^\infty h(n)t^n = \frac{4}{\pi^2} K(4\sqrt{t})^2, \]
which is again easy to prove using the series for $K(t)$. Recall that $h(n)$ is related to the moments of $K(x)K'(x)$, and thus we have:
\begin{thm}
\begin{equation}\label{gf}
\int_0^1 \frac{x}{1-t^2x^2} K(x)K'(x) \, \dx = \frac{\pi}4 K(t)^2.
\end{equation}
\end{thm}

Equation (\ref{gf}) seems to be a remarkable extension of its (much easier) cousins,
\begin{equation} \label{gf2}
\int_0^1 \frac{1}{1-t^2x^2} K'(x) \, \dx = \frac{\pi}2 K(t), \quad \int_0^1 \frac{1}{1-t^2x^2} E'(x) \, \dx = \frac{\pi}{2t^2} (K(t)-E(t)).
\end{equation}

Note that manipulations of (\ref{gf}, \ref{gf2}) give myriads of integrals, we list some of which below (below $G$ denotes \textit{Catalan's constant}):
\begin{eqnarray*}
\int_0^1 \frac{\arctan(x)}{x} K'(x) \, \dx & = & \pi G, \\
\int_0^1 \frac{2}{x} K(x)K'(x)(K(x)-E(x)) \, \dx & = & \int_0^1 K(x)^2E'(x) \, \dx, \\
\int_0^1 \,_2F_1\left({{1,\frac{n+1}2}\atop{\frac{n+3}2}}\bigg|x^2\right) \,xK(x)K'(x) \, \dx & = & \frac{(n+1)\pi}{4} \int_0^1 t^n K(t)^2 \, \mathrm{d}t. \\
\end{eqnarray*}
The last identity specialises to 
\[ \int_0^1 \frac{-\log(1-x^2)}{x}K(x)K'(x) \, \dx =  \frac78 \pi \zeta(3). \]

\section{Two complementary elliptic integrals} \label{two1}

In \cite{zud}, Zudilin's Theorem connects, as a special case, triple integrals of rational functions over the unit cube with generalized hypergeometric functions ($_7F_6$'s). We state a restricted form of the theorem which is sufficient for our purposes here: 

\begin{thm}[Zudilin] \label{zudilin} Given $h_0, \ldots, h_5$ for which both sides converge,

\begin{eqnarray}
& &  \int_0^1 \int_0^1 \int_0^1 \frac{x^{h_2-1}y^{h_3-1}z^{h_4-1}(1-x)^{h_0-h_2-h_3} (1-y)^{h_0-h_3-h_4} (1-z)^{h_0-h_4-h_5}}{(1-x(1-y(1-z)))^{h_1}} \, \mathrm{d}x \, \mathrm{d}y \, \mathrm{d}z \nonumber \\
& = & \frac{\Gamma(h_0+1)\prod_{j=2}^4 \Gamma(h_j) \prod_{j=1}^4 \Gamma(h_0+1-h_j-h_{j+1})}{\prod_{j=1}^5 \Gamma(h_0+1-h_j)}\times \nonumber \\
& & _7F_6\left( {{h_0,1+h_0/2,h_1,h_2,h_3,h_4,h_5}\atop{h_0/2,1+h_0-h_1,1+h_0-h_2,1+h_0-h_3,1+h_0-h4,h+h_0-h_5}}\bigg| 1 \right).
\end{eqnarray}
\end{thm}

In \cite{walks2}, this theorem is applied to derive closed form evaluations for moments of random walks from their integral representations. 

The idea here is to work backwards, writing a single integral involving products of elliptic integrals as a double, then a triple integral of the required form, and then apply Theorem \ref{zudilin}.  To do so, we require the following formulae, which can be readily verified (see e.g. \cite{prud}):

\begin{eqnarray}
\int_0^1 \frac{\dx}{\sqrt{x(1-x)(a-x)}} & = & \frac{2}{\sqrt{a}} K\left(\frac{1}{\sqrt{a}}\right), \\
\int_0^1 \sqrt{\frac{a-x}{x(1-x)}} \dx & = & 2\sqrt{a} \ E\left(\frac{1}{\sqrt{a}}\right), \label{eb}\\
\int_a^1 \frac{\mathrm{d}y}{\sqrt{y(1-y)(y-a)}} & = & 2 K'(\sqrt a), \\
\int_a^1 \frac{\sqrt{y}}{\sqrt{(1-y)(y-a)}} \, \mathrm{d} y& = & 2 E'(\sqrt a). \label{ea}
\end{eqnarray}

Though the simple cases in this section (corresponding to $n=0$) are tabulated in \cite{handbook}, the general results appear to be new.

Using the above relations, we have, for instance,
\begin{eqnarray*}
\int_0^1 E'(y)^2 \, dy & = &  \frac12 \int_0^1 \int_{a^2}^1 \sqrt{\frac{y}{(1-y)(y-a^2)}} E(\sqrt{1-a^2}) \, \mathrm{d}y \mathrm{d}a \\
 & = & \frac14 \int_0^1 \int_0^1 \sqrt{\frac{y}{(1-y)z(1-z)}} E(\sqrt{1-yz}) \, \mathrm{d}y\mathrm{d}z \\
 & = & \frac18 \int_0^1 \int_0^1 \int_0^1 \sqrt{\frac{y(1-yz)}{(1-y)z(1-z)}}\sqrt{\frac{\frac{1}{1-yz}-x}{x(1-x)}} \, \dx \mathrm{d} y \mathrm{d} z \\
 & = & \frac18 \int_0^1 \int_0^1 \int_0^1 \sqrt{\frac{y(1-x(1-y(1-z)))}{x(1-x)(1-y)z(1-z)}} \, \dx \mathrm{d} y \mathrm{d} z.\\
\end{eqnarray*}

The first equality follows from (\ref{ea}), the second from the change of variable $a^2 \mapsto y z$, the third from (\ref{eb}),  and the fourth from $z \mapsto 1-z$. Now Theorem \ref{zudilin} applies to the last integral. This strategy works for the rest of this section, so we are led to:

\begin{prop} For all $n$ where the integral converges,
\end{prop}
\begin{equation} \int_0^1 x^n E'(x)^2 \, \dx = \frac{2^{4n}(n+1)^3(n+3)^2}{16(n+2)^3(n+4)}\frac{\Gamma\left(\frac12(n+1)\right)^8}{\Gamma(n+1)^4} \, _7F_6\left({ {-\frac12, \frac12, \frac12, \frac32, \frac{n+3}2, \frac{n+3}2, \frac{n+7}4}\atop{1, \frac{n+3}4, \frac{n+2}2, \frac{n+4}2, \frac{n+4}2, \frac{n+6}2} } \bigg| 1\right). \end{equation}

When $n$ is odd, the hypergeometric reduces to known constants only involving $\zeta(3)$, which we prove in Theorem \ref{zeta3} below.

Similarly, by building up the $E'$ integral then  $K'$, we obtain:

\[ \int_0^1 E'(x)K'(x) \, \dx = \frac18 \int_0^1 \int_0^1 \int_0^1 \sqrt{\frac{1-x(1-y(1-z))}{x(1-x)y(1-y)z(1-z)}} \, \dx \mathrm{d} y \mathrm{d} z, \]

Hence, we have:

\begin{prop} For all $n$ where the integral converges, \end{prop}

\begin{equation} \int_0^1 x^n E'(x)K'(x) \, \dx = \frac{2^{4n}(n+1)^2}{16(n+2)}\frac{\Gamma\left(\frac12(n+1)\right)^8}{\Gamma(n+1)^4} \, _7F_6\left({ {-\frac12, \frac12, \frac12, \frac12, \frac{n+1}2, \frac{n+1}2, \frac{n+5}4}\atop{1, \frac{n+1}4, \frac{n+2}2, \frac{n+2}2, \frac{n+2}2, \frac{n+4}2} } \bigg| 1\right).\end{equation}

Alternatively,  by building up the $K'$ integral then $E'$, we obtain:

\[ \int_0^1 E'(x)K'(x) \, \dx = \frac18 \int_0^1 \int_0^1 \int_0^1 \frac{\sqrt y}{\sqrt{x(1-x)(1-y)z(1-z)(1-x(1-y(1-z)))}} \, \dx \mathrm{d} y \mathrm{d} z, \]

So the next proposition allows us to produce some equalities for the $_7F_6$'s involved:

\begin{prop} For all $n$ where the integral converges, \end{prop}

\begin{equation} \int_0^1 x^n E'(x)K'(x) \, \dx = \frac{2^{4n}(n+1)^3(n+3)}{16(n+2)^3}\frac{\Gamma\left(\frac12(n+1)\right)^8}{\Gamma(n+1)^4} \, _7F_6\left({ {\frac12, \frac12, \frac12, \frac32, \frac{n+3}2, \frac{n+3}2, \frac{n+7}4}\atop{1, \frac{n+3}4, \frac{n+2}2, \frac{n+4}2, \frac{n+4}2, \frac{n+4}2} } \bigg| 1\right).\end{equation}

Finally, we also have

\[ \int_0^1 K'(x)^2 \, \dx = \frac18 \int_0^1 \int_0^1 \int_0^1 \frac{1}{\sqrt{x(1-x)y(1-y)z(1-z)(1-x(1-y(1-z)))}} \, \dx \mathrm{d} y \mathrm{d} z, \]

Therefore:

\begin{prop} For all $n$ where the integral converges, \end{prop}

\begin{equation} \int_0^1 x^n K'(x)^2 \, \dx = \frac{2^{4n}(n+1)}{16}\frac{\Gamma\left(\frac12(n+1)\right)^8}{\Gamma(n+1)^4} \, _7F_6\left({ {\frac12, \frac12, \frac12, \frac12, \frac{n+1}2, \frac{n+1}2, \frac{n+5}4}\atop{1, \frac{n+1}4, \frac{n+2}2, \frac{n+2}2, \frac{n+2}2, \frac{n+2}2} } \bigg| 1\right).\end{equation}

Again, the odd moments are computable in terms of $\zeta(3)$, with the particularly simple 
\begin{equation} \int_0^1 x K'(x)^2 \dx = \frac{7}{4} \zeta(3). \end{equation}

\begin{remark} \label{oddr} \rm{Therefore, all the odd moments of $K'^2,E'^2, K'E'$ have particularly simple forms. By using (\ref{gen}), we can iteratively obtain \emph{all the odd moments} of $K^2, E^2, KE$. As an example, $\int_0^1 x^3 K(x)^2 \, \dx = \frac18(2+7\zeta(3))$.} \end{remark}

Now we prove the observation regarding the appearance of $\zeta(3)$:

\begin{thm} \label{zeta3} When $n$ is odd, the $n$th moment of $K'^2, E'^2, K'E', K^2, E^2$ and $KE$ is expressible as $a+b\zeta(3)$, where $a,b \in \mathbb{Q}$.
\end{thm}

\begin{proof} We prove the case for the pair $K'^2$ and $K^2$; the other 2 pairs are similar. 

Firstly, when $n$ is odd, the summand of the $_7F_6$ for $K'^2$ is a rational function: 
\begin{equation} \label{partial}
\frac{(2k+m+1)(k+1)^2(k+2)^2\cdots (k+m)^2}{(k+1/2)^4(k+3/2)^4\cdots (k+m+1/2)^4}, \end{equation}
for we have ignored any rational constants and wrote $n=2m+1$. We can explictly sum (\ref{partial}) and verify the statement of the theorem for the first few moment of $K'^2$. By using (\ref{gen}), we can likewise do this for $K^2$.

We now use the recursion (\ref{recc}) to see that the statement holds for all odd moments of $K^2$. By using (\ref{gen}) (but with the role of $x$ and $x'$ reversed), we obtain the result for $K'^2$.

\end{proof}

\begin{remark} \label{partfrac} \rm{ We sketch another proof by expanding (\ref{partial}) into partial fractions.

As each partial fraction has at most a 4th power on the denominator, the sum can involve at most $\zeta(2), \zeta(3), \zeta(4)$, and some contribution from the linear denominators.

Because the linear denominator terms must converge, their sum must eventually telescope, and hence contribute only a rational number.

We recall that partial fractions can be obtained via a derivative process akin to computing Taylor series coefficients; indeed, if we write 
\[ \frac{f(x)}{(x-a)^n} = \frac{A_1}{(x-a)} + \frac{A_2}{(x-a)^2} + \cdots + \frac{A_n}{(x-a)^n} ,\]
then $A_n = f(a), A_{n-1} = f'(a)/1!, \ldots, A_1 = f^{(n-1)}(a)/(n-1)!$.

When applied to (\ref{partial}), it is easy to check that, when $n \equiv 3$ mod 4, the presence of the numerator $2k+m+1$ makes the terms with quadratic and quartic denominators telescope out, leaving us with rational numbers (these terms occur in pairs related by the transformation $k \mapsto (-n-1)/2-k$, where said linear numerator switches sign). Similarly, the terms with cubic denominators double.

When $n \equiv 1$ mod 4, $2k+m+1$ cancels out with one of the denominators, making it a cubic; it can be similarly checked that its partial fraction can have no quadratic term: this is equivalent to showing (\ref{partial}) with all powers of $2k+m+1$ removed has 0 derivative at $k=(-n-1)/4$, which is true because it is symmetric around that point.

Therefore only the cubic terms remain, giving us $\zeta(3)$. 

This type of partial fraction argument is at the heart of the result that infinitely many odd zeta values are irrational (see \cite{zeta}, which, incidentally, is the motivation for Zudilin's Theorem).  }
\end{remark}

We note in passing that by either one of the two known transformations for non-terminating $_7F_6$'s (\cite{bailey}), we can write our $_7F_6$ as the sum of two $_4F_3$'s with $\Gamma$ pre-factors, where one term is easily reduced to closed form when $n$ is odd, while the harder term becomes  obtainable in light of our reduction.

\section{Sporadic results} \label{sporadic}

We list some results found by ad-hoc methods; some are not moment evaluations while others are preparatory for later sections.

\subsection{Explicit primitives}

A small number of integrals have explicit primitives; we list some here: 
\[ K(\sqrt x), \, \frac{E(\sqrt x)}{1-x}, \, \frac{E(x)}{1 \pm x},  \, \frac{E(x)}{1-x^2},\,\frac{K(\sqrt x)}{\sqrt x},  \,K(\sqrt[4]{x}),  \,\frac{K(\sqrt[4]{x})}{\sqrt x},  \,\frac{K(\sqrt[4]{x})}{\sqrt[4]{x}}. \]
We note in passing that even some good CAS may struggle with finding all primitives for this short list.

\subsection{Complementary modulus}

Using the obvious transformation $x \mapsto x'$, we have 
\begin{equation}
 \int_0^1 x K(x)^a E(x)^b K'(x)^c E'(x)^d \dx = \int_0^1 x K'(x)^a E'(x)^b K(x)^c E(x)^d \dx,
\end{equation}
More generally:
\begin{equation}\label{gen}
 \int_0^1 x^{2n+1} K(x)^a E(x)^b K'(x)^c E'(x)^d \dx = \int_0^1 x(1-x^2)^n K'(x)^a E'(x)^b K(x)^c E(x)^d \dx,
\end{equation}
and for any function $f$,
\[ \int_0^1 x x' f(x) \, \dx = \int_0^1 x^2 f(x') \, \dx. \]

\subsection{Imaginary argument}

In \cite{prud} volume III, some integrals with the imaginary argument $ix$ are considered, e.g.
\[ \int_0^1 x K'(x) K(ix) \, \dx = \frac12 G \pi. \]
This can be proven by expanding $x K(ix)$ as a series and summing the moments of $K'(x)$. Other evaluations are done similarly. For instance, we can easily obtain recursions for the moments of $K(ix), E(ix)$.

We record here that 
\[ E(i \sqrt{k}) = \sqrt{k+1} \ E(\sqrt{k/(k+1)}), \ \mathrm{and} \ K(i \sqrt{k}) = 1/\sqrt{k+1} \ K(\sqrt{k/(k+1)}).\]

\subsection{Quadratic transforms}

Using the quadratic transforms (\ref{k1}, \ref{k2}), we obtain
\begin{eqnarray} \label{quad}
\int_0^1 K(x)^n \,\dx & = & \frac12 \int_0^1 K'(t)^n \left(\frac{1+t}{2}\right)^{n-2}\, \mathrm{d}t, \nonumber \\
\int_0^1 K'(x)^n \,\dx & = & 2 \int K(t)^n (1+t)^{n-2} \,\mathrm{d}t.
\end{eqnarray}

Setting  $n=1$ we get the known special case 
\[ \int_0^1 \frac{K(x)}{x+1} \, \dx = \frac{\pi^2}{8}. \]

Using a cubic transform of the Borweins (\cite{cubic}), this has been generalized in \cite{catalan}. The appropriate generalization of (\ref{1}) is
\[ \int_0^1 \, _2F_1\left({{\frac13, \frac23}\atop{1}}\bigg| 1-x^3 \right)^2 \dx = 3 \int_0^1 \, _2F_1\left({{\frac13, \frac23}\atop{1}}\bigg| x^3 \right)^2 \dx. \]
Equation (\ref{1}) itself is obtained by setting $n=2$ in (\ref{quad}).

Using (\ref{k2}) on the integrand $x K(x)^3$, we get 
\[ \int_0^1 2(1-x)K(x)^3 \,\dx = \int_0^1 x K(x)^3 \,\dx, \]
so combined with (\ref{quad}), we deduce
\begin{equation} \label{k3}
 \int_0^1 K'(x)^3 \,\dx = \frac{10}{3} \int_0^1 K(x)^3 \,\dx = 5 \int_0^1 x K(x)^3 \,\dx = 5 \int_0^1 x K'(x)^3 \,\dx. 
\end{equation}

Using (\ref{e1}, \ref{e2}), we have 
\[ \int_0^1 E(x)^n \frac{2^{n+1}}{(x+1)^{n+2}} \,\dx = \int_0^1 (E'(x)+xK'(x))^n \,\dx; \]
when $n=1$ we can evaulate this in closed form:
\[ \int_0^1 \frac{E(x)}{(x+1)^3} \,\dx = \frac{\pi^2}{32}+\frac14. \]
The case $n=2$ gives a messier closed form, in view of our knowledge of the moments.

Similarly, it follows from  (\ref{e1}, \ref{e2}) that
\[ \int_0^1 E'(x)^n \frac{2^{n+1}}{(x+1)^{n+2}}\, \dx = \int_0^1 (2E(x)-(1-x^2)K(x))^n \,\dx; \]
with the special case
\[ \int_0^1 \frac{E'(x)}{(x+1)^3}\, \dx = \frac{G}{8}+\frac{5}{16}. \]

\subsection{Relation to random walks}

In \cite{walks2}, many relations are derived while computing $W_4(n)$, the $n$th moment of the distance from the origin of a 4-step uniform random walk on the plane. For instance, we have:

\[ W_4(1) = \frac{16}{\pi^3} \int_0^1 (1-3x^2)K'(x)^2 \, \dx. \]

In \cite{walks3}, the following identities are proven, via connections with  Meijer G-functions:

\[ \frac{\pi^3}{4} W_4(-1) = \int_0^{\pi/2} K(\sin t)^2 \sin t \, \mathrm{d}t = 2 \int_0^{\pi/2} K(\sin t)^2 \cos t \, \mathrm{d}t = \int_0^{\pi/2} K(\sin t)K(\cos t) \, \mathrm{d}t, \] and
\begin{equation} \int_0^1 K(x)K'(x)x' \, \dx = \int_0^1 K(x)^2 \, \dx. \end{equation}

\section{Fourier series} \label{fs}

As recorded in \cite{bbbg}, we have the following Fourier (sine) series valid on $(0,2\pi)$:
\begin{lem}
\begin{equation} \label{fourier}
K(\sin t) = \sum_{n=0}^\infty  \frac{\Gamma(n+1/2)^2}{\Gamma(n+1)^2}  \sin((4n+1)t).
\end{equation}
\end{lem}

For completeness, we sketch a proof here:

\begin{proof}
By symmetry we see that only the coefficients of $\sin((2n+1)t)$ are non-zero. Indeed, by a change of variable $\cos t \mapsto x$, the coefficients are
\[ \frac{4}{\pi} \int_0^1 K'(x) \frac{\sin((2n+1)t)}{\sin t} \, \mathrm{d}t. \]
The fraction in the integrand is precisely $U_{2n}(x)$, where $U_n(x)$ denotes the \textit{Chebyshev polynomial of the second kind} (\cite{wolf}), given by
\[ U_{2n}(x) = \sum_{k=0}^n (-1)^k \binom{2n-k}{k} (2x)^{2n-2k}. \]
We now interchange summation and integration, and use the moments of $K'$. The resulting coefficient contains a $_3F_2$, which after a transformation (\cite{bailey}, section 3.2) becomes amenable to Saalsch{\"u}tz's  theorem (\cite{aar}), and we obtain (\ref{fourier}).
\end{proof}

We may apply the same method to obtain a Fourier sine series for $E(\sin t)$ valid on $(0,\pi)$, which we had not been able to locate in the literature. In mirroring the last step, the resulting $_3F_2$ is reduced to the closed form below using Sister Celine's method:
\begin{lem}
\begin{equation} \label{fourier2}
E(\sin t) = \sum_{n=0}^\infty \frac{\Gamma(n+1/2)^2}{2\Gamma(n+1)^2}\sin((4n+1)t) + \sum_{n=0}^\infty \frac{(n+1/2)\Gamma(n+1/2)^2}{2(n+1)\Gamma(n+1)^2}\sin((4n+3)t).
\end{equation}
\end{lem}

Using Parseval's formula on (\ref{fourier}) and (\ref{fourier2}) gives 
\begin{eqnarray} & & \int_0^{\pi/2} K(\sin t)^2 \, \mathrm{d}t = \nonumber 2 \int_0^{\pi/2} K(\sin t)E(\sin t) \, \mathrm{d}t \\
&= &\int_0^1 \frac{K(x)^2}{\sqrt{1-x^2}} \, \dx = 2 \int_0^1 \frac{E(x)^2}{\sqrt{1-x^2}} \, \dx \nonumber \\
& = & 2\int_0^1 K(x)K'(x) \, \dx = \frac{\pi^3}{4} \, _4F_3\left({{\frac12,\frac12,\frac12,\frac12}\atop{1,1,1}}\bigg| 1 \right). \label{fou}
\end{eqnarray}
We also get $\int_0^{\pi/2} K(\sin t)^2 \cos(4t) \, \mathrm{d}t$ as a sum of three $_4F_3$'s, and $\int_0^{\pi/2} E(\sin t)^2 \, \mathrm{d}t$ as a sum of four $_4F_3$'s. Section 3.7 of \cite{bbbg} provides more identities with more exotic arguments, as well as connections with Meijer G-functions.

Experimentally we  found that 
\begin{eqnarray}
\int_0^{\pi/2} \frac14 K(\sin t)^2 \sin 4t \, \mathrm{d}t &= & \int_0^1 K(x)^2(x-2x^3) \, \dx = \int_0^1 x K(x)((1-x^2)K(x)-E(x)) \, \dx \nonumber \\
& = &  \int_0^1 K'(x)^2(2x^3-x) \, \dx = \int_0^1 x K'(x)(x^2 K'(x)-E'(x)) \, \dx \nonumber \\
& =& \int_0^1 x K^2(x)-2x E(x) K(x) \, \dx= -\frac12.
\end{eqnarray}

The last line, an evaluation which is somewhat surprising, is equivalent to
\[ \int_0^1 xK(x)^2+2xE(x)^2-3xK(x)E(x) \, \dx = 0, \]
which is routine to check as we know all the odd moments.

Inserting a factor of $\cos(t)^2$ before squaring the Fourier series (\ref{fourier}) and integrating, and we are led to 
\[ \int_0^1 x' K(x)^2 \, \dx = \frac{\pi^3}{16}\left(2 \, _3F_2\left({{\frac12,\frac12,\frac12,\frac12}\atop{1,1,1}}\bigg| 1\right)-1\right) = \int_0^1 \frac{\sqrt{x}}{x+1}K'(x)^2 \, \dx. \]

The Fourier series (\ref{fourier}) combined with a quadratic transform gives:
\begin{eqnarray}
& & \int_0^{\pi/2} K(\sin t) \, \mathrm{d}t = \int_0^1 \frac{K(x)}{\sqrt{1-x^2}} \, \dx =  \int_0^1 \frac{K'(x)}{\sqrt{1-x^2}} \, \dx =  \int_0^1 \frac{K(x)}{\sqrt{x}} \, \dx \nonumber \\
& = & \frac12 \int_0^1 \frac{K'(x)}{\sqrt{x}} \, \dx = K\left(\frac{1}{\sqrt2}\right)^2 = \frac{1}{16\pi}\Gamma\left(\frac14\right)^4.
\end{eqnarray}
This result has been generalized in (\cite{catalan}).

\section{Legendre's relation}\label{lege}

Legendre's relation (\cite{agm}) is related to the Wronskian of $K$ and $E$, and shows that the two integrals are closely coupled:

\begin{equation} \label{legendre}
E(x)K'(x)+E'(x)K(x)-K(x)K'(x) = \frac{\pi}{2}.
\end{equation}

If we take Legendre's relation,  multiply both sides by $K'(x)$ and integrate, then replacing known moments by their closed forms, we arrive at
\[ \int_0^1 3 E'(x)K'(x)K(x)-K(x)K'(x)^2 \, \dx = \frac{\pi^3}{8}. \]
Similarly, if we had multiplied by $K(x)$, the result would be
\begin{equation} \int_0^1 3 E(x)K(x)K'(x) - 2 K(x)^2K'(x) \, \dx = \int_0^1 2E'(x)K(x)^2-E(x)K(x)K'(x) \, \dx = \pi G. \end{equation}

Using closed form of the moments, we also have: 
\[ \int_0^1 2xE'(x)K(x)^2K'(x)-xK(x)^2K'(x)^2 \, \dx = \frac{\pi^4}{32}, \]
\[ \int_0^1 2xE'(x)^2K(x)E(x) - x K(x)K'(x)E(x)E'(x) \, \dx = \frac{\pi^2}{16}+\frac{\pi^4}{128}. \]

We can multiply Legendre's relation by any function whose integral vanishes on the interval $(0,1)$ to produce another relation. Suitable candidates for the function include $x(K(x)-K'(x)), 2(2K'(x)-3E'(x)), 2E'(x)-K'(x), 2E(x)-K(x)-1$, and a vast range of polynomials.  For instance one could obtain
\[ \int_0^1 2E'^2(x) K(x)+2E(x)E'(x)K(x)-5E'(x)K(x)K'(x)+K(x)K'(x)^2 \, \dx = 0. \]
Unfortunately, we are not able to uncouple any of the above sums and differences to obtain a closed form for the integral of a single product.

\section{Integration by parts} \label{parts}

The following simple but fruitful idea is crucial to this section. We look at the derivative $(1-x^2)^n \frac{\mathrm{d}}{\dx} (x^k K(x)^a E(x)^b K'(x)^c E'(x)^d)$, and  integrate by parts to yield
\begin{eqnarray} 
& & \int_0^1 (1-x^2)^n \frac{\mathrm{d}}{\dx} \Bigl( x^k K(x)^a E(x)^b K'(x)^c E'(x)^d \Bigr) \, \dx \nonumber \\
& = & \int_0^1 2nx(1-x^2)^{n-1} x^k K(x)^a E(x)^b K'(x)^c E'(x)^d \, \dx, \label{byparts}
\end{eqnarray}
(where an offset by a constant is possible if the function is a powers of $E$ or $E'$).

In practice, we take $n,k \in \{0,1,2\}$ to produce the cleanest identities. Later, we will also explore when $n$ is a half-integer, as well as replacing $1-x^2$ by $1-x$ in (\ref{byparts}).

\subsection{Table for products of two elliptic integrals}

We now systematically analyse the tables kindly provided by D. H. Bailey, the construction of which are described in \cite{combat}. The tables contain all known (in fact, almost certain \textit{all}) linear relations for integrals of products of up to $k$ elliptic integrals ($k \le 6$) and a polynomial in $x$ with degree at most 5. In this subsection we exclusively look at when $k=2$.

We use the derivative $x \frac{\mathrm{d}}{\dx}E(x)^2 = 2 E(x)^2-2E(x)K(x)$, and integrate by parts to deduce
\begin{equation}\label{eq1} \int_0^1 3 E(x)^2 - 2 E(x)K(x) \, \dx= 1. \end{equation}

More generally, 
\begin{equation}  1 = (n+k+1) \int_0^1 x^k E(x)^n \, \dx - n \int_0^1 x^k E(x)^{n-1} K(x) \, \dx. \end{equation}

Two special cases of the above are prominent:
\begin{equation} \label{eq2} \int_0^1 5 x^2 E(x)^2 - 2 E(x)K(x) \, \dx = 1, \end{equation}
and
\begin{equation} \label{rec1}  \int_0^1 (n+2)x^{n-1}E(x)^2 - 2 x^{n-1} E(x)K(x) \, \dx = 1. \end{equation}

The derivative of $K(x)E(x)$ (via integration by parts) gives
\begin{equation} \label{eq3} \int_0^1 (1-3x^2)E(x)K(x)+E(x)^2-(1-x^2)K(x)^2 \, \dx = 0, \end{equation}
while more generally
\begin{equation} \label{rec2} \int_0^1 n x^{n-1}E(x)K(x)-(n+2)x^{n+1}E(x)K(x)+x^{n-1}E(x)^2-x^{n-1}K(x)^2+x^{n+1}K(x)^2 \, \dx = 0.
\end{equation}

The derivative of $K(x)^2$ produces
\begin{equation} \label{eq4} \int_0^1 (1+x^2) K^2(x) \, \dx = 2 \int_0^1 K(x)E(x) \, \dx, \end{equation}
and more generally,
\begin{equation} \label{rec3} \int_0^1 2x^{n-1}E(x)K(x)+(n-2)x^{n-1}K(x)^2-nx^{n+1}K(x)^2 \, \dx = 0. \end{equation}

The derivative of $E'(x)^2$ gives
\[ \int_0^1 2xE'(x)^2-xE'(x)K'(x)\dx = \int_0^1 2xE(x)^2-xE(x)K(x) \,\dx =\frac12. \]

The derivative of $K'(x)^2$ gives 
\[ \int_0^1 2K'(x)E'(x)-(1-x^2)K'(x)^2 \, \dx = 0,\] reconfirming a result from random walks (\cite{walks2}), which was first proven  in a much more roundabout way via a non-trivial group action on the integrand.

The derivative of $E'(x)K'(x)$ gives
\[ \int_0^1 (1-3x^2)E'(x)K'(x) \,\dx = \int_0^1 E'(x)^2-x^2 K'(x)^2 \, \dx,\]
which when combined with the last result, gives
\[ \int_0^1 (1+3x^2)E'(x)K'(x) \, \dx = \int_0^1 K'(x)^2-E'(x)^2 \, \dx. \]

The derivative of $K(x)K'(x)$ gives
\[ \int_0^1 x^2 K(x)K'(x)+K(x)E'(x)-K'(x)E(x) \, \dx = 0,  \]
which, when combined with Legendre's relation (\ref{legendre}), results in
\[ \int_0^1 2 E'(x)K(x) - (1-x^2) K(x)K'(x) \, \dx = \frac{\pi}{2}. \]

Our results above and in the previous sections actually provide direct proofs of most entries in the table where the polynomial involved is of degree at most 1. In fact, it would simply be a matter of tenacity to prove many other entries in the table where the polynomial is of higher degree. As an example, we indicate how to prove one of the few entries that requires more work:
\begin{equation} \int_0^1 E(x)(3E'(x)-K'(x)) \, \dx = \frac{\pi}{2}. \end{equation}

We note that the left hand side can be written as two $_4F_3$'s. We combine their summands into a single term and simplify; the result can be summed explicitly by Gosper's algorithm, and the correct limit on the right hand side follows.

There is only one entry in the tables for which we do not possess a proof:
\begin{conj}
\begin{equation} \label{eq5} \int_0^1 2K(x)^2 - 4 E(x)K(x) + 3 E(x)^2 - K'(x)E'(x) \, \dx \stackrel{?[1]}{=} 0.
\end{equation}
\end{conj}

We note that, among moments of products of two elliptic integrals, there are only five that we do not have closed forms of:
\[ E(x)^2, \,x^2 E(x)^2, \,E(x)K(x), \,x^2 E(x)K(x),\, x^2K(x)^2, \]
as all the odd moments are known, and the other even moments may be obtained from these ignition values. Unfortunately, we can only prove four equations connecting them, namely (\ref{eq1}, \ref{eq2}, \ref{eq3}, \ref{eq4}). A proof of (\ref{eq5}) would give us enough information to solve for all five moments.

\subsection{Recurrences for the moments}

As already hinted in Lemma \ref{pi3}, the moments enjoy recurrences with polynomial coefficients. As an example, by combining (\ref{rec1}, \ref{rec2}, \ref{rec3}), we may obtain, with $K_n := \int_0^1 x^n K(x)^2 \, \dx$,
\begin{equation} \label{recc}
(n+1)^3 K_{n+2}-2n(n^2+1)K_n+(n-1)^3K_{n-2}=2.
\end{equation}
This then shows that $K_n$ is a rational number plus a rational multiple of $\zeta(3)$ for odd $n$.

Similarly, recurrences for other products may be obtained in the same way, though the linear algebra becomes more prohibitive. We are able to obtain, as another example, for $E_n := \int_0^1 x^n E(x)^2 \, \dx$,
\begin{equation}
(n+1)(n+3)(n+5)E_{n+2}-2(n^3+3n^2+n+1)E_n+(n-1)^3E_{n-2}=8,
\end{equation}
while the recursion for the moments of $EK$ is straightforward from this using (\ref{rec1}). The recursion for the moments of $K'^2$ is identical that (\ref{recc}) except the right hand side is 0, not 2. In particular, this approach is used in the proof of Theorem \ref{zeta3}.

\subsection{More results}

We also discovered some results not found in the tables by incorporating constants such as $\pi$ and $G$ into the search space. Below we highlight some other pretty formulae. Using $(1-x^2)\frac{\mathrm{d}}{\dx}(x^2 K(x)^2)$ gives
\[ \int_0^1 x^3 K(x)^2 - x K(x)E(x) \, \dx = 0, \]
while the derivative of $ x^2 K(x)K'(x)$ gives 
\begin{eqnarray} & &\int_0^1 x K(x)K'(x) \, \dx = \int_0^1 2x^3K(x)K'(x) \, \dx \nonumber \\
&=&  \int_0^1 \frac{1-x}{1+x}K(x)K'(x) \, \dx = \frac{\pi^3}{16}.
\end{eqnarray}

The derivative of $x^2K(x)^n$ implies
\[ \int_0^1 x^{n-1}K(x)^{n-1}(nE(x)-2x^2K(x)) \, \dx = 0. \]

The derivative of $x^{2n}K(x)K'(x)$ together with Legendre's relation (\ref{legendre}) gives
\begin{equation} \label{leg}
 \int_0^1 x^{2n-1}(E'(x)K(x)+n(x^2-1)K(x)K'(x)) \, \dx = \frac{\pi}{8n}. 
\end{equation}

In (\ref{byparts}),we can take $n = \frac12$. For instance, the derivative of $xx'K(x)$ gives 
\[ \int_0^1 \frac{E(x)}{x'} \, \dx = \int_0^1 \frac{x^2 K(x)}{x'} \, \dx, \] 
and the derivative of $x x' K(x)^2$ recaptures (\ref{fou}).

The derivative of $x'K(x)$ gives \[ \int_0^1 \frac{K(x)-E(x)}{xx'} \, \dx = \frac{\pi}2, \] note that each part does not converge.

In fact, \begin{equation} K(x)-E(x) = \frac{\pi x^2}{4} \, _2F_1\left({{\frac12,\frac32}\atop{2}}\bigg| x^2 \right), \end{equation}
therefore
\[ \int_0^1 \frac{K(x)-E(x)}{x} \, \dx = \frac{\pi}2-1, \ \int_0^1 \frac{K(x)-E(x)}{x^2} \, \dx = 1, \ \int_0^1 \frac{K(x)-E(x)}{x x'} \, \dx = \frac{\pi}2. \]

The derivative of $x(1-x)K(x)^2$ gives
\[ \int_0^1 \frac{2K(x)E(x)}{x+1} \, \dx = \int_0^1 K(x)^2 \, \dx, \]
while that of $x(1-x)E(x)K(x)$ gives
\[ \int_0^1 \frac{E(x)^2}{x+1} \, \dx = \int_0^1 (2x-1)E(x)K(x)+(x-1)K(x)^2 \, \dx. \]

Collecting what we know about the integral of $K(x)^2$, we have the following:
\begin{thm} 
\begin{eqnarray*}
 \int_0^1 K(x)^2 \, \dx & = & \frac12 \int_0^1 K'(x)^2 \, \dx \\
& = & \int_0^1 K'(x)^2 \frac{x}{x'} \, \dx \\
& = & \int_0^1 K(x)K'(x)x' \, \dx \\
& = & \int_0^1 \frac{2K(x)E(x)}{x+1} \, \dx \\
& = & \frac{2}{\pi} \int_0^1 \frac{\arcsin x}{\sqrt{1-x^2}} K(x)K'(x) \, \dx \\
& = & \frac{4}{\pi} \int_0^1 \mathrm{arctanh}(x)K(x)K'(x) \, \dx.
\end{eqnarray*}
\end{thm}
\begin{proof} The last two equalities follow from (\ref{gf}); the rest had been proven elsewhere.
\end{proof}

\subsection{Table for products of three elliptic integrals}

We now consider the linear relations involving the product of three elliptic integrals ($k=3$ in the tables). As the number of relations provided by Bailey's tables is huge, we restrict most of our attention to a class of integrals that turn out to be mutually related by a rational factor.

Below we tabulate all the products for which neat integrals may be deduced by differentiating them and integrating by parts:

\begin{center}
\begin{tabular}{|l|rl|} \hline
Product: & Integral: & \\ \hline
$K(x)^3$  & $\int_0^1 K(x)^3-3K(x)^2E(x) \, \dx$ & $= 0$ \\
$K(x)^2 K'(x)$ & $ \int_0^1 K(x)^2 E'(x)+K(x)^2K'(x)-2K(x)K'(x)E(x) \, \dx$&$ = 0$ \\
$K'(x)^2 K(x)$ & $\int_0^1 E(x)K'(x)^2 - 2 E'(x)K(x)K'(x) \, \dx $&$= 0$ \\ 
$K'(x)^3$ & $\int_0^1 K'(x)^3 - 3 K'(x)^2 E'(x)\, \dx $&$= 0$ \\ 
$E'(x)^3$ & $\int_0^1 5 x E'(x)^3 - 3 x E'(x)^2 K'(x) \, \dx $&$= 1$ \\
$E(x)^3$ & $\int_0^1 4E(x)^3-3E(x)^2K(x) \, \dx $&$= 1$ \\ \hline
\end{tabular}
\end{center}

We now prove 
\[ \int_0^1 K(x)^2K'(x) \, \dx = \frac23 \int_0^1 K(x)K'(x)^2 \, \dx. \]
We make a change variable $x \mapsto (1-x)/(1+x)$ to the left hand side, use a quadratic transform, then apply $x \mapsto 2\sqrt{x}/(1+x)$ to part of the result followed by another quadratic transform (\ref{k1}, \ref{k2}), obtaining
\[ \int_0^1 3x K(x)K'(x)^2 \, \dx = \int_0^1 K(x) K'(x)^2 \, \dx, \]
now combining the pieces proves the claim.

If we make the change of variable $x \mapsto (1-x)/(1+x)$, then apply (\ref{e1}), we have
\[ \int_0^1 \frac{K(x)^2 E(x)}{1+x} \, \dx = \frac49 \int_0^1 K(x)^3 \, \dx. \]

Using the derivative of $x^2(1-x)K(x)^3$, we can show that $\int_0^1 x K(x)^2 E(x)/(x+1)\,  \dx$ is also linearly related to the above integral.

Therefore, gathering the results in this section and equation(\ref{k3}), we have determined:

\begin{thm} \label{equiv} Any two integrals in each of the following groups are related by a rational factor:
\begin{equation*}
K(x)^3, K'(x)^3, xK(x)^3, xK'(x)^3, K(x)^2 E(x), K'(x)^2 E'(x), \frac{K^2(x) E(x)}{1+x}, \frac{xK^2(x) E(x)}{1+x};
\end{equation*}

\begin{equation}
K(x)K'(x)^2, K(x)^2 K'(x), x K(x)K'(x)^2, x K(x)^2 K'(x).
\end{equation}
\end{thm}

We cannot yet, however, show that the two groups themselves are related by a rational factor, though it is true numerically to extremely high precision. In fact, the Inverse Symbolic Calculator gives the remarkable evaluation:

\begin{conj} \label{conj2}
\begin{equation} \int_0^1 K'(x)^3 \, \dx \stackrel{?[2]}{=} 2 K\left(\frac{1}{\sqrt2}\right)^4 = \frac{\Gamma(1/4)^8}{128\pi^2}. \end{equation}
\end{conj}

Once proven, this also gives explicit closed forms for the integrals of $E'(x)K'(x)K(x)$, $E(x)K'(x)K(x)$, and $E'(x)K(x)^2$, for we can relate each of these to one of the above and a constant (by the results of Section \ref{lege}).

In view of Theorem \ref{equiv}, equations (\ref{kk1}) and (\ref{altf}), interchanging the order of summation and integration gives an equivalent form of Conjecture \ref{conj2}:
\begin{eqnarray}
& & \sum_{n=0}^\infty \frac{8}{(2n+1)^2} \, _4F_3\left({{\frac12,\frac12,n+1,n+1}\atop{1,n+\frac32,n+\frac32}}\bigg|1\right) \nonumber \\
& = & \sum_{n=0}^\infty \frac{\Gamma(n+1/2)^4}{\Gamma(n+1)^4} \, _4F_3\left({{\frac12,\frac12,-n,-n}\atop{1,\frac12-n,\frac12-n}}\bigg|1\right) \stackrel{?[2]}{=} \frac{\Gamma(1/4)^8}{24\pi^4}. 
\end{eqnarray}

\subsection{Products of four elliptic integrals and conclusion}

If we take the derivative for $K(x)^4$, use the integral (\ref{quad}) connecting $K'(x)^4$ and $K(x)^4$, plus a quadratic transform, then we obtain
\begin{equation} \int_0^1 24E(x)K(x)^3 - 8 K(x)^4 - K'(x)^4 \, \dx = 0, \end{equation}
which is one of the first non-trivial identities in Bailey's table for $k=4$. Many more tabulated relations for products of three and four elliptic integrals can be proven, albeit the complexity of the proofs increase. As perceptively noted in \cite{combat}, 
\begin{quotation} ``[it] seems to be more and more the case as experimental computational tools improve, our ability to discover outstrips our ability to prove.''
\end{quotation}

\bigskip
\paragraph{Acknowledgements.} The author wishes to thank David H. Bailey for providing extensive tables showing relations of various integrals, and J. M. Borwein and W. Zudilin for many helpful comments.

\end{document}